\documentclass{amsart}
\usepackage[all]{xy}
\usepackage{color}
\usepackage{amsthm}
\usepackage{amssymb}
\usepackage[colorlinks=true]{hyperref}

\numberwithin{equation}{section}

\newtheorem{theorem}[equation]{Theorem}
\newtheorem*{theorem*}{Theorem}
\newtheorem{lemma}[equation]{Lemma}

\newtheorem*{conjecture*}{Mamma Conjecture}
\newtheorem*{conjecture1*}{Mamma Conjecture (revisited)}

\newtheorem*{corollary*}{Corollary}

\theoremstyle{remark}

\newtheorem{example}[equation]{Example}
\newtheorem{non-example}[equation]{Non-Example}

\theoremstyle{remark}
\newtheorem{remark}[equation]{Remark}

\newcommand{\cA}{{\mathcal A}}
\newcommand{\cB}{{\mathcal B}}
\newcommand{\cC}{{\mathcal C}}
\newcommand{\cD}{{\mathcal D}}

\newcommand{\cF}{{\mathcal F}}
\newcommand{\cG}{{\mathcal G}}

\newcommand{\cO}{{\mathcal O}}

\newcommand{\bbC}{\mathbb{C}}

\newcommand{\bbF}{\mathbb{F}}

\newcommand{\bbL}{\mathbb{L}}

\newcommand{\bbP}{\mathbb{P}}

\newcommand{\bbQ}{\mathbb{Q}}

\newcommand{\bbZ}{\mathbb{Z}}

\DeclareMathOperator{\Var}{Var} 
\DeclareMathOperator{\Mot}{Mot} 
\DeclareMathOperator{\Sym}{Sym} 
\DeclareMathOperator{\Alt}{Alt} 

\DeclareMathOperator{\SmProj}{SmProj} 
\DeclareMathOperator{\Chow}{Chow} 
\DeclareMathOperator{\Corr}{Corr} 
\DeclareMathOperator{\CHM}{CHM} 
\DeclareMathOperator{\KM}{KM} 

\DeclareMathOperator{\id}{id}

\DeclareMathOperator{\K}{K}

\DeclareMathOperator{\Mix}{KMM} 
\DeclareMathOperator{\KPM}{KPM} 
\DeclareMathOperator{\Hmo}{Hmo} 

\newcommand{\uk}{\underline{k}}

\newcommand{\dgcat}{\mathsf{dgcat}}
\newcommand{\perf}{\cD^{\dg}_\mathsf{perf}}

\newcommand{\dg}{\mathsf{dg}}

\newcommand{\Hom}{\mathsf{Hom}}

\newcommand{\op}{\mathsf{op}}

\newcommand{\too}{\longrightarrow}

\newcommand{\ie}{\textsl{i.e.}\ }

\begin{document}

\title[Chow motives versus non-commutative motives]{Chow motives versus non-commutative motives}
\author{Gon{\c c}alo~Tabuada}

\address{Gon\c calo Tabuada, Departamento de Matematica, FCT-UNL, Quinta da Torre, 2829-516 Caparica,~Portugal }
\email{tabuada@fct.unl.pt}

\date{\today}


\thanks{The author was partially supported by the grant {\tt PTDC/MAT/098317/2008}.}

\begin{abstract}
In this article we formalize and enhance Kontsevich's beautiful insight that Chow motives can be embedded into non-commutative ones after factoring out by the action of the Tate object. We illustrate the potential of this result by developing three of its manyfold applications: (1) the notions of Schur and Kimura finiteness admit an adequate extension to the realm of non-commutative motives; (2) Gillet-Soul{\'e}'s motivic measure admits an extension to the Grothendieck ring of non-commutative motives; (3) certain motivic zeta functions admit an intrinsic construction inside the category of non-commutative motives.
\end{abstract}

\maketitle
\vskip-\baselineskip
\vskip-\baselineskip
\vskip-\baselineskip

\section{Introduction}
In the early sixties Grothendieck envisioned the existence of a ``universal'' cohomology theory of algebraic varieties. Among several conjectures and developments, a contravariant functor
$$ M: \SmProj^\op \too \Chow_\bbQ$$
from smooth projective varieties (over a base field) towards a certain category of {\em Chow motives} was constructed; see \S\ref{sub:Chow} for details. Intuitively, $\Chow_\bbQ$ encodes all the geometric/arithmetic information about smooth projective varieties and acts as a gateway between algebraic geometry and the assortment of numerous cohomology theories (de Rham, $l$-adic, crystalline, and others); for details consult the monograph~\cite{Jannsen}.

During the last two decades Bondal, Drinfeld, Kaledin, Kapranov, Kontsevich, Van den Bergh, and others, have been promoting a broad non-commutative geometry program in which geometry is performed directly on triangulated categories and its differential graded enhancements; see \cite{Kapranov,BB,Drinfeld,Chitalk,Kaledin,IAS,ENS,Miami,finMot}. One of the beauties of this program is its broadness. It encompasses several research fields such as algebraic geometry, representation theory of quivers, sympletic geometry, and even mathematical physics, making it a cornerstone of modern mathematics. In analogy with the commutative world, a central problem is the development of a theory of non-commutative motives. By adapting the classical notions of smoothness and properness to the non-commutative world, Kontsevich introduced recently a category $\Mix$ of {\em non-commutative motives}; see \S\ref{sub:Konts} for details. The following question is therefore of major importance.  

\medbreak
\noindent\textbf{Question:}\textit{ How to bridge the gap between Grothendieck's category of Chow motives and Kontsevich's category of non-commutative motives ?}
\medbreak

This question was settled by Kontsevich in \cite[\S4.1.3]{finMot}. The purpose of this article is to formalize and enhance his beautiful insight while illustrating its power through three applications.

Recall from \S\ref{sub:Chow} that the category $\Chow_\bbQ$ is $\bbQ$-linear, additive and symmetric monoidal. Moreover, it is endowed with an important $\otimes$-invertible object, the Tate motive $\bbQ(1)$. The functor $-\otimes \bbQ(1)$ is an automorphism of $\Chow_\bbQ$ and so we can consider the associated orbit category $\Chow_\bbQ\!/_{\!\!-\otimes\bbQ(1)}$; consult \S\ref{sec:orbit} for the precise construction. Informally speaking, Chow motives which differ from a Tate twist become equal in the orbit category. 

In the non-commutative world we have the category $\Hmo$ of dg categories up to derived Morita equivalence and a universal functor $U: \Hmo \to \Mot$, with values in a triangulated category, that sends short exact sequences to distinguished triangles; see \S\ref{sec:non-mot} for details. Recall that a dg category $\cA$ is {\em smooth} and {\em proper} in the sense of Kontsevich if its complexes of morphisms are perfect and $\cA$ is perfect as a bimodule over itself. As explained in \S\ref{sub:Konts}, Kontsevich's category $\Mix$ of non-commutative motives can be identified with the smallest thick triangulated subcategory of $\Mot$ spanned by the objects $U(\cA)$, with $\cA$ a smooth and proper dg category. By first taking rational coefficients $(-)_\bbQ$ and then passing to the idempotent completion $(-)^\natural$ we obtain in particular a natural inclusion $\Mix_\bbQ^\natural \subset \Mot_\bbQ^\natural$; see \S\ref{sub:rat-coef}-\ref{sub-idemp}.

The gap between algebraic varieties and dg categories can be bridged by associating to every smooth projective variety $X$ its dg category $\perf(X)$ of perfect complexes of $\cO_X$-modules; see Lunts-Orlov \cite{LO} or \cite[Example~4.5]{CT1}.
\begin{theorem}[Kontsevich]\label{thm:main}
There exists a fully-faithful, $\bbQ$-linear, additive, and symmetric monoidal functor $R$ making the following diagram commutative
\begin{equation}\label{eq:diag-main}
\xymatrix@C=2em@R=1.5em{
\SmProj^\op \ar[rr]^{\perf(-)} \ar[d]_M && \Hmo \ar[d]^U \\
\Chow_\bbQ \ar[d]_\pi && \Mot \ar[d]^{(-)_\bbQ^\natural} \\
\Chow_\bbQ\!/_{\!\!-\otimes \bbQ(1)} \ar[rr]_-R && \Mix_\bbQ^\natural \subset \Mot_\bbQ^\natural \qquad \qquad
}
\end{equation}
\end{theorem}
Intuitively speaking, Theorem~\ref{thm:main} formalizes the conceptual idea that the commutative world can be embedded into the non-commutative world after factoring out by the action of the Tate object. In the next three sections we will illustrate the potential of this result by describing some of its manyfold applications.
\section{Schur and Kimura finiteness}
By construction, the category $\Mix_\bbQ^\natural$ is $\bbQ$-linear, idempotent complete and moreover symmetric monoidal. Hence, given a partition $\lambda$ of a natural number $n$, we can consider the associated Schur functor $\mathbf{S}_\lambda: \Mix_\bbQ^\natural \to \Mix_\bbQ^\natural$ which sends a non-commutative motive $N$ to the direct summand of $N^{\otimes n}$ determined by $\lambda$; see Deligne's foundational work~\cite{Deligne}. We say that $N$ is {\em Schur-finite} if it is annihilated by some Schur functor. When $\lambda=(n)$, resp. $\lambda=(1,\ldots,1)$, the associated Schur functor $\Sym^n:=\mathbf{S}_{(n)}$, resp. $\Alt^n:=\mathbf{S}_{(1,\ldots,1)}$, should be considered as the motivic analogue of the usual $n^\mathrm{th}$ symmetric, resp. wedge, product functor of $\bbQ$-vector spaces. We say that $N$ is evenly, resp. oddly, finite dimensional if $\Alt^n(N)$, resp. $\Sym^n(N)$, vanishes for some $n$. Finally, $N$ is {\em Kimura-finite} if it admits a direct sum decomposition $N\simeq N_+\oplus N_-$ into a evenly $N_+$ and oddly $N_-$ finite dimensional object. Note that Kimura-finiteness implies Schur finiteness. 

In the commutative world, the aforementioned general finiteness notions were extensively studied by Andr{\'e}-Kahn, Guletskii, Guletskii-Pedrini, Kimura, Mazza, and others; see~\cite{Andre,AK,Guletskii,GP,Kimura,Mazza}. For instance, Guletskii and Pedrini proved that given a smooth projective surface $X$ (over a field of characteristic zero) with $p_g(X)=0$, the Chow motive $M(X)$ is Kimura-finite if and only if Bloch's conjecture on the Albanese kernel for $X$ holds.

Theorem~\ref{thm:main} establishes a precise relationship between the finiteness conditions on the commutative and non-commutative worlds. For simplicity, let us write $NC$ for the composed functor $U(\perf(-))^\natural_\bbQ: \SmProj^\op \to \Mix_\bbQ^\natural$.
\begin{theorem}\label{thm:2}
Let $X$ be a smooth projective variety.
\begin{itemize}
\item[(i)] If $M(X)$ is Schur-finite, resp. Kimura-finite, in $\Chow_\bbQ$, then $NC(X)$ is Schur-finite, resp. Kimura-finite, in $\Mix_\bbQ^\natural$. 
\item[(ii)] The converse for Schur-finiteness holds, \ie if $NC(X)$ is Schur-finite in $\Mix_\bbQ^\natural$, then $M(X)$ is Schur-finite in $\Chow_\bbQ$.
\end{itemize}
\end{theorem} 
Informally speaking, Theorem~\ref{thm:2} shows that Schur-finiteness is a notion which is insensitive to commutativity. A potential application of this fact is the use of non-commutative methods in order to prove Schur-finiteness for certain Chow motives. In what concerns Kimura-finiteness, recall from \cite{Andre} that all the Chow motives of the form $M(X)$, with $X$ an abelian variety, are Kimura-finite. Using Theorem~\ref{thm:2} and the stability of Kimura-finiteness under several constructions (see \cite{Guletskii,Mazza}), we then obtain a large class of examples of Kimura-finite non-commutative motives.
\section{Motivic measures}
Let $k$ be a field and $\Var_k$ the category of algebraic $k$-varieties. Recall that the {\em Grothendieck ring $K_0(\Var_k)$ of algebraic $k$-varieties} is the abelian group generated by the isomorphism classes of objects in $\Var_k$ modulo the scissor congruence relations $[X]=[Z] + [X\setminus Z]$, where $Z$ is a closed subvariety of $X$. Multiplication is given by the fiber product over $\mathrm{spec}(k)$. Although very important, the structure of this ring remains rather misterious. In order to capture some of its flavor several {\em motivic measures}, \ie ring homomorphisms $\mu:K_0(\Var_k) \to A$ towards commutative rings, have been built; see \cite{Looijenga}. For instance, Gillet-Soul{\'e}~\cite{GS1} constructed for any field $k$ an interesting motivic measure $\mu_{GS}: K_0(\Var_k) \to K_0(\Chow_\bbQ)$ with values in the Grothendieck ring of Chow motives.

Recall that by construction $\Mix_\bbQ^\natural$ is a $\otimes$-triangulated category. Hence, we can consider its Grothendieck ring\footnote{A precise relationship between $K_0(\Mix)$ (and hence between $K_0(\Mix_\bbQ^\natural)$) and To{\"e}n's secondary $K$-theory was described in \cite[\S8.4]{CT1}.} $K_0(\Mix_\bbQ^\natural)$ defined in the standard way. Theorem~\ref{thm:main} enables then the extension of Gillet-Soul{\'e}'s motivic measure to the non-commutative world.
\begin{theorem}\label{thm:3}
The assignment $X \mapsto [NC(X)] \in K_0(\Mix_\bbQ^\natural)$, for every smooth projective $k$-variety $X$, 
gives rise to a well-defined motivic measure $$\mu_{NC}: K_0(\Var_k) \to K_0(\Mix_\bbQ^\natural)\,.$$ 
\end{theorem} 
\begin{remark}
Making use of Bittner-Looijenga's~\cite{Bittner,Looijenga} presentation of $K_0(\Var_k)$ in the case of a field $k$ of characteristic zero, Bondal-Larsen-Lunts constructed in \cite{BLL} a motivic measure $\mu_{BLL}$ with values in a certain ring of pre-triangulated dg categories. Using arguments similar to those of \cite[Prop.~7.3]{IMRN}, it can be shown that in this particular case the motivic measure $\mu_{NC}$ factors through $\mu_{BLL}$.
\end{remark}
Intuitively speaking, $\mu_{NC}$ measures the information concerning algebraic varieties (modulo the scissor congruence relations) which can be recovered from their derived categories of perfect complexes.
\begin{example}\label{ex:main}
Assume that $k=\bbC$. Then, the assignment 
$$X \mapsto \chi_c(X):=\sum_{i \geq 0} (-1)^i \mathrm{dim}\,H^i(X, \Omega_X)\,,$$
for every smooth projective $\bbC$-variety $X$, gives rise to a well-defined $\bbZ$-valued motivic measure $\mu_{\chi_c}$; see \cite{Andre1}. This information can be completely recovered from the category of perfect complexes and so $\mu_{\chi_c}$ factors through $\mu_{NC}$; see \S\ref{sec:proofs}.
\end{example}
However, as the following non-example illustrates, the passage from the commutative to the non-commutative world forgets some geometric/arithmetic information.
\begin{non-example}\label{non-ex}
Assume that $k=\bbF_q$, where $q=p^n$ with $p$ a prime number and $n$ a positive integer. Then, the assignment $X \mapsto \# X(\bbF_q)$, for every smooth projective $\bbF_q$-variety $X$, gives rise to a well-defined $\bbZ$-valued motivic measure $\mu_{\#}$; see \cite{Andre1}. In contrast with $\mu_{\chi_c}$, this motivic measure does {\em not} factor through $\mu_{NC}$. Take for instance $X=\bbP^1$. Beilinson's~\cite{Be} semi-orthogonal decomposition of $\perf(\bbP^1)$ into two copies of $\perf(\bbF_q)$ implies that $NC(\bbP^1)=U(\bbF_q)_\bbQ^\natural \oplus U(\bbF_q)_\bbQ^\natural$. Hence, since $U(\bbF_q)_\bbQ^\natural$ is the unit of the $\otimes$-triangulated category $\Mix_\bbQ^\natural$, we conclude that $[NC(\bbP^1)]=2$ in $K_0(\Mix_\bbQ^\natural)$. This implies that any motivic measure $\mu$ which factors through $\mu_{NC}$ verifies the equality $\mu(\bbP^1)=2$, but this is clearly not the case for $\mu_\#$ since $\#(\bbP^1)=q+1$. Intuitively speaking, from the non-commutative motive $NC(X)$ we can only recover the number of points of $X$ modulo $q-1$.
\end{non-example}

\section{Motivic zeta functions}
Recall from Kapranov~\cite{Kapranov1}, that the {\em motivic zeta function} of a smooth algebraic $k$-variety $X$ with respect to a motivic measure $\mu$ is the following formal power series
$$ \zeta_\mu(X;t) := \sum_{n=0}^\infty \mu([S^n(X)])t^n \in A[[t]]\,,$$
where $S^n(X)$ denotes the $n^{\textrm{th}}$ symmetric power of $X$. For instance, when $\mu$ is the motivic measure of the above Non-Example~\ref{non-ex} we recover the classical Hasse-Weil zeta function.

On the other hand, recall that by construction $\Mix_\bbQ^\natural$ is an idempotent complete $\bbQ$-linear category. Hence, given any non-commutative motive $N$, we can construct its zeta function intrinsically inside $\Mix_\bbQ^\natural$ as follows
$$ \zeta(N;t):= \sum_{n=0}^\infty [\Sym^n(N)]t^n \in K_0(\Mix_\bbQ^\natural)[[t]]\,.$$
Theorem~\ref{thm:main} allows us then to compare this intrinsic construction with the motivic zeta function associated to the motivic measure $\mu_{NC}$.
\begin{theorem}\label{thm:4}
Let $X$ a smooth projective $k$-variety. Then
$$\zeta(NC(X);t) = \zeta_{\mu_{NC}}(X;t)\in K_0(\Mix_\bbQ^\natural)[[t]]\,.$$
\end{theorem}
As explained in \cite{Andre1}, the motivic zeta function of a smooth projective $k$-variety $X$ with respect to the motivic measure of Example~\ref{ex:main} is given by $(1-t)^{-\chi_c(X)}$. Theorem~\ref{thm:4} combined with Example~\ref{ex:main} show us then that, in contrast with the Hasse-Weil zeta function, this information can be completely recovered solely from the intrinsic zeta function associated to the non-commutative motive $NC(X)$.
\section{Some categories of motives}
In this section we recall, and adapt to our convenience, the construction of some categories of motives following Grothendieck, Manin, and Gillet-Soul{\'e}. These will be used in the proof of Theorem~\ref{thm:main}. In what follows, $k$ denotes a (fixed) base field.
\subsection{Grothendieck's category of Chow motives}{(consult \cite{Scholl})}\label{sub:Chow}
Given a smooth projective $k$-variety $X$ and an integer $d$, we will write $\mathcal{Z}^d(X)$ for the $d$-codimensional cycle group of $X$ and $A^d(X):=(\mathcal{Z}^d(X)\otimes_\bbZ \bbQ)/\mathrm{(rational~equivalence)}$ for the $d$-codimensional Chow group with rational coefficients of $X$. If $Z$ is a cycle on $X$, we will denote by $[Z]$ its class in $A^d(X)$. Let $X$ and $Y$ be two smooth projective $k$-varieties, $X=\amalg_i X_i$ the decomposition of $X$ in its connected components, and $d_i$ the dimension of $X_i$. Then, $\Corr^r(X,Y):=\oplus_i A^{d_i+r}(X_i \times Y)$ is called the {\em space of correspondances of degree $r$} from $X$ to $Y$. Given $ f \in \Corr^r(X,Y)$ and $g\in \Corr^s(Y,Z)$ their composition $g \circ f\in \Corr^{r+s}(X,Z)$ is defined by the classical formula
\begin{equation}\label{eq:comp-corr}
g \circ f := (\pi_{XY})_\ast (\pi_{XY}^\ast(f) \cdot \pi^{\ast}_{YZ}(g))\,.
\end{equation}

The category $\Chow_\bbQ$ of {\em Chow motives} (with rational coefficients) is defined as follows: its objects are the triples $(X,p,m)$ where $X$ is a smooth projective $k$-variety, $m$ is an integer, and $p=p^2 \in \Corr^0(X,X)$ is an idempotent endomorphism; its morphisms are given by
$$ \Hom_{\Chow_\bbQ}((X,p,m),(Y,q,n)):= q \circ \Corr^{n-m}(X,Y) \circ p\,;$$
composition is induced by the above composition \eqref{eq:comp-corr} of correspondances. By construction, the category $\Chow_\bbQ$ is $\bbQ$-linear, additive and pseudo-abelian (\ie every idempotent endomorphism has a kernel). Moreover, it carries a symmetric monoidal structure defined on objects by the formula $(X,p,m) \otimes (Y,q,n):= (X \times Y, p\otimes q, m+n)$. The unit object for this symmetric monoidal structure is the Chow motive $(\mathrm{spec}(k), \id, 0)$, where $\id=[\Delta]$ is the class of the diagonal $\Delta$ in $\Corr^0(\mathrm{spec}(k), \mathrm{spec}(k))$. The {\em Tate motive} $(\mathrm{spec}(k), \id, 1)$ will be denoted by $\bbQ(1)$. Note that $\bbQ(1)$ is a $\otimes$-invertible object. Finally, we have a natural symmetric monoidal functor
\begin{eqnarray}\label{eq:func-M}
M: \SmProj^\op \too \Chow_\bbQ && X \mapsto (X, \id, 0) 
\end{eqnarray}
which maps a morphism $f:X \to Y$ in $\SmProj$ to $[\Gamma^t_f]$, where $\Gamma^t_f$ is the transpose of the graph $\Gamma_f=\{(x,f(x))\,|\, x \in X\} \subset X \times Y$ of $f$.
\subsection{Manin's category of motives}{(consult \cite{Manin})}
The category $\CHM_\bbQ$ of {\em Manin's motives} (with rational coefficients) is defined as follows: its objects are the pairs $(X,p)$, where $X$ is a smooth projective $k$-variety and $p^2=p \in \Corr^0(X,X)$ is an idempotent endomorphism; its morphisms are given by 
$$ \Hom_{\CHM_\bbQ}((X,p),(Y,q)):= q \circ \bigoplus_{j \in \bbZ} \Corr^j(X,Y) \circ p\,;$$
composition is induced by the above composition \eqref{eq:comp-corr} of correspondences. Similarly to $\Chow_\bbQ$, the category $\CHM_{\bbQ}$ is $\bbQ$-linear, additive, pseudo-abelian, and symmetric monoidal. Moreover, we have a natural symmetric monoidal functor
\begin{eqnarray}\label{eq:func-Manin}
\SmProj^\op \too \CHM_\bbQ&& X \mapsto (X,\id)\,.
\end{eqnarray}
\subsection{Gillet-Soul{\'e}'s category of motives}{(consult \cite{GS1,GS2})}\label{sub:GS}
The category $\KM_\bbQ$ of {\em Gillet-Soul{\'e}'s motives} (with rational coefficients) is constructed\footnote{Gillet-Soul{\'e} did not considered the pseudo-abelianization procedure.} in two steps. First, consider the category $KM_\bbQ$ whose objects are the smooth projective $k$-varieties and whose morphisms are given by 
$$ \Hom_{KM_\bbQ}(X,Y):= K_0(X \times Y)_\bbQ=K_0(X \times Y)\otimes_\bbZ \bbQ\,.$$
Composition is defined as follows:
$$
\begin{array}{ccc}
K_0(X \times Y)_\bbQ \times K_0(Y\times Z)_\bbQ  & \too &  K_0(X \times Z)_\bbQ \\
([\cF], [\cG]) & \mapsto & \sum_{i \geq 0} (-1)^i [\mathrm{Tor}_i^{\cO_Y}(\cF,\cG)]\,.
\end{array}
$$
Then, take the pseudo-abelianization of $KM_\bbQ$. The objects of the resulting category $\KM_\bbQ$ are the pairs $(X,p)$, where $X$ is a smooth projective $k$-variety and $p^2=p \in K_0(X \times X)_\bbQ$ is an idempotent endomorphism. Morphisms are given by 
$$ \Hom_{\KM_\bbQ}((X,p),(Y,q)):= q \circ K_0(X \times Y)_\bbQ \circ p$$
and composition is induced from the one on $KM_\bbQ$. In particular, we have a natural fully-faithful functor $KM_\bbQ \to \KM_\bbQ$. By construction, $\KM_\bbQ$ is $\bbQ$-linear, additive and pseudo-abelian, Moreover, it carries a symmetric monoidal structure defined on objects by the formula $(X,p)\otimes (Y,q)=(X \times Y, p \otimes q)$. Finally, we have a natural symmetric monoidal functor
\begin{eqnarray}\label{eq:func-GS}
\SmProj^\op \too \KM_\bbQ && X \mapsto (X, \id)
\end{eqnarray}
which maps a morphism $f: X \to Y$ in $\SmProj$ to $[\cO_{\Gamma_f^t}] \in K_0(Y\times X)_\bbQ$.
\section{Non-commutative motives}\label{sec:non-mot}
A {\em differential graded (=dg) category}, over a fixed commutative base ring $k$, is a category enriched over 
cochain complexes of $k$-modules (morphisms sets are such complexes) in such a way that composition fulfills the Leibniz rule\,: $d(f\circ g)=(df)\circ g+(-1)^{\textrm{deg}(f)}f\circ(dg)$; see Keller's ICM address~\cite{ICM}. 

As proved in \cite{IMRN}, the category $\dgcat$ of dg categories carries a Quillen model structure whose weak equivalences are the {\em derived Morita equivalences}, \ie the dg functors $F:\cA \to \cB$ which induce an equivalence $\cD(\cA) \stackrel{\sim}{\to} \cD(\cB)$ on derived categories (see \cite[\S4.6]{ICM}). The homotopy category obtained is denoted by $\Hmo$. 

All the classical invariants such as cyclic homology (and its variants), algebraic $K$-theory, and even topological cyclic homology, extend naturally from $k$-algebras to dg categories. In order to study all these invariants simultaneously the {\em universal localizing invariant} of dg categories was constructed in~\cite{Duke}. Roughly speaking, it consists of a functor $U: \Hmo \to \Mot$, with values in a triangulated category, that sends short exact sequences (\ie sequences of dg categories which become exact after passage to the associated derived categories; see \cite[\S4.6]{ICM}) to distinguished triangles
\begin{eqnarray*}
0 \to \cA \to \cB \to \cC \to 0 & \mapsto & U(\cA) \to U(\cB) \to U(\cC) \to U(\cA)[1]
\end{eqnarray*}
and which is universal\footnote{The formulation of the precise universal property makes use of the language of Grothendieck derivators; see \cite[\S10]{Duke}.} with respect to this property. All the mentioned invariants are localizing and so they factor (uniquely) through $U$. The tensor product of $k$-algebras extends naturally to dg categories, giving rise to a symmetric monoidal structure on $\Hmo$. The corresponding unit is the dg category $\uk$ with a single object and with the base ring $k$ as the dg algebra of endomorphisms. Using Day's (derived) convolution product, the symmetric monoidal structure on $\Hmo$ was extended to $\Mot$ making the functor $U$ symmetric monoidal; see \cite{CT1}.
\subsection{Kontsevich's category}\label{sub:Konts}
A dg category $\cA$ is called {\em proper} if for each ordered pair of objects
$(x,y)$ in $\cA$ the cochain complex of $k$-modules $\cA(x,y)$ is perfect, and {\em smooth} if it is perfect as a bimodule over itself. As explained in \cite{IAS}, Kontsevich's construction of the category $\Mix$ of non-commutative motives decomposes in three steps:
\begin{itemize}
\item[(i)] first, consider the category $\KPM$ (enriched over spectra) whose objects are the smooth and proper dg categories and whose morphisms from $\cA$ to $\cB$ are given by the $K$-theory spectrum $K(\cA^\op \otimes \cB)$. Composition corresponds to the (derived) tensor product of bimodules.
\item[(ii)] then, take the formal triangulated envelope of $\KPM$. Objects in this new category are formal finite extensions of formal shifts of objects in $\KPM$.
\item[(iii)] finally, add formal direct summands for idempotent endomorphisms and pass to the underlying homotopy category. The resulting category $\Mix$ is in particular triangulated and its morphisms are given in terms of $K$-theory groups. 
\end{itemize}
In \cite[Prop.~8.5]{CT1}, Kontsevich's construction was characterized in a simple and elegant way: $\Mix$ identifies with the thick triangulated subcategory of $\Mot$ spanned by the objects $U(\cA)$, with $\cA$ a smooth and proper dg category. In particular, $\Mix \subset \Mot$. For smooth and proper dg categories $\cA$ and $\cB$, we have
\begin{equation}\label{eq:key}
\Hom_{\Mix}(U(\cA), U(\cB)) \simeq K_0(\cA^\op \otimes \cB)
\end{equation}
with composition given by the (derived) tensor product of bimodules. Finally, note that since smooth and proper dg categories are stable under tensor product, $\Mix$ is moreover a $\otimes$-triangulated subcategory of $\Mot$.
\subsection{Rational coefficients}\label{sub:rat-coef}
In this subsection we assume that $k$ is a field. Let $\Mix_\bbQ$, resp. $\Mot_\bbQ$, be the category obtained form $\Mix$, resp. from $\Mot$, by tensoring each abelian group of morphisms with $\bbQ$. By construction, we have natural functors
\begin{eqnarray}\label{eq:localization}
(-)_\bbQ : \Mix \too \Mix_\bbQ && (-)_\bbQ : \Mot \too \Mot_\bbQ\,.
\end{eqnarray}

\begin{lemma}
The categories $\Mix_\bbQ$ and $\Mot_\bbQ$ inherit from $\Mix$ and $\Mot$, respectively, a canonical $\otimes$-triangulated structure making the natural functors \eqref{eq:localization} $\otimes$-triangulated.
\end{lemma}
\begin{proof}
Since both cases are similar we focus on the category $\Mot_\bbQ$. The dg category $\uk$ is smooth and proper and so the above isomorphism \eqref{eq:key} restricts to
\begin{equation*}
\Hom_{\Mot}(U(\uk), U(\uk)) \simeq K_0(k)\,.
\end{equation*}
Since by hypothesis $k$ is a field, we conclude that the endomorphism ring of the unit $U(\uk)$ of the $\otimes$-triangulated category $\Mot$ is $K_0(k)=\bbZ$. Hence, by applying \cite[Thm.~3.6]{Balmer-SSS} to the multiplicative set $S:=\bbZ \setminus \{0\} \subset K_0(k)$ of the endomorphism ring of $U(\uk)$ we obtain a canonical $\otimes$-triangulated structure on $\Mot_\bbQ$ making the functor $(-)_\bbQ$ $\otimes$-triangulated.
\end{proof}
\subsection{Idempotent completion}\label{sub-idemp}
The $\otimes$-triangulated categories $\Mix_\bbQ$ and $\Mot_\bbQ$ of the previous section are not idempotent complete. Thanks to Balmer-Schlichting \cite{BS}, their idempotent completions $\Mix_\bbQ^\natural$ and $\Mot_\bbQ^\natural$ carry a canonical triangulated structure. By construction they carry also a canonical $\otimes$-structure. Hence, we obtain natural $\otimes$-triangulated functors
\begin{eqnarray}
(-)^\natural: \Mix_\bbQ \too \Mix_\bbQ^\natural && (-)^\natural: \Mot_\bbQ \too \Mot_\bbQ^\natural\,.
\end{eqnarray} 
\section{Orbit categories}\label{sec:orbit}
Let $\cC$ be an additive category and $F: \cC \to \cC$ an automorphism (a standard construction allow us to replace an autoequivalence by an automorphism). By definition, the {\em orbit category} $\cC/F$ has the same objects as $\cC$ and morphisms
$$ \Hom_{\cC/F}(X,Y) := \bigoplus_{j \in \bbZ} \Hom_\cC(X, F^jY)\,.$$
Composition is induced by the one on $\cC$. The canonical projection functor $\pi: \cC \to \cC/F$ is endowed with a natural isomorphism $\pi \circ F \stackrel{\sim}{\Rightarrow} \pi$ and is universal among all such functors. By construction, $\cC/F$ is still an additive category and the projection is an additive functor.

Let us now suppose that $\cC$ is moreover endowed with a symmetric monoidal structure such that the tensor product $-\otimes-$ is bi-additive.
\begin{lemma}\label{lem:tec}
If the automorphism $F$ is of the form $-\otimes \cO$, with $\cO$ a $\otimes$-invertible object in $\cC$, then the orbit category $\cC/_{\!\!-\otimes \cO}$ inherits a natural symmetric monoidal structure making the projection functor $\pi$ symmetric monoidal.
\end{lemma}
\begin{proof}
Let us start by defining the tensor product on $\cC/_{\!\!-\otimes \cO}$. On objects it is the same as the one on $\cC$. On morphisms let
$$ \bigoplus_{j \in \bbZ} \Hom_\cC(X, Y \otimes \cO^j) \times \bigoplus_{j \in \bbZ} \Hom_\cC(Z, W \otimes \cO^j) \too \bigoplus_{j \in \bbZ} \Hom_\cC(X\otimes Z, (Y\otimes W) \otimes \cO^j)$$
be the unique bilinear morphism which sends the homogeneous maps $f: X \to Y \otimes \cO^r$ and $g: Z \to W \otimes \cO^s$ to the homogeneous map
$$ X\otimes Z \stackrel{f\otimes g}{\too} Y \otimes \cO^r \otimes W \otimes \cO^s \stackrel{\tau}{\too} (Y\otimes W) \otimes \cO^{(r+s)}\,,$$
where $\tau$ is the commutativity isomorphism constraint. The associativity and commutativity isomorphism constraints are obtain from the corresponding ones of $\cC$ by applying the projection functor $\pi$. A routine verification shows that these definitions endow the orbit category $\cC/_{\!\!-\otimes \cO}$ with a symmetric monoidal structure making the projection functor $\pi$ symmetric monoidal.
\end{proof}  
\section{Proof of Theorem~\ref{thm:main}}
\begin{proof}
The bulk of the proof consists on constructing functors $\theta_1, \theta_2$ and $\theta_3$ making the following diagram commutative
\begin{equation*}\label{eq:main-diag}
\xymatrix{
\SmProj^\op \ar[d]_M \ar@{=}[r] & \SmProj^\op \ar[dd]^{\eqref{eq:func-Manin}} \ar@{=}[r] & \SmProj^\op \ar[dd]^{\eqref{eq:func-GS}} \ar[r]^-{\perf(-)} & \Hmo \ar[d]^U \\
\Chow_\bbQ \ar[d]_\pi & &  & \Mot \ar[d]^{(-)_\bbQ^\natural} \\
\Chow_\bbQ\!/_{\!\!-\otimes\bbQ(1)} & \CHM_\bbQ \ar[l]^-{\theta_1} & \KM_\bbQ \ar[l]^-{\theta_2}  \ar[r]_-{\theta_3} & \Mix_\bbQ^\natural \subset \Mot_\bbQ^\natural \qquad \qquad
}
\end{equation*}
Let us start with the functor $\theta_1$. Given Chow motives $(X,p,m)$ and $(Y,q,n)$ the description of the orbit category given in the previous section show us that
$$ \Hom_{\Chow_\bbQ\!/_{\!\!-\otimes\bbQ(1)}}((X,p,m),(Y,q,n))=q \circ \bigoplus_{j \in \bbZ} \Corr^{(n+j)-m}(X,Y) \circ p\,.$$
Hence, we define $\theta_1$ to be the functor that maps $(X,p)$ to $(X,p,0)$ and which is the identity on morphisms. By construction, it is clearly fully-faithful. Notice that the identity of $X$, \ie the class of the diagonal $\Delta$ in $\Corr^0(X,X)$, can be considered as an isomorphism in $\Chow_\bbQ\!/_{\!\!-\otimes\bbQ(1)}$ between $(X,p,0)$ and $(X,p,m)$. This implies that $\theta_1$ is moreover essentially surjective and so an equivalence of categories. Note that $\theta_1$ makes the left rectangle in the above diagram commutative. Moreover, it is $\bbQ$-linear, additive and symmetric monoidal.

Let us now focus our attention on the functor $\theta_2$. Similarly to $\KM_\bbQ$, the category $\CHM_\bbQ$ is obtained as the pseudo-abelianization of a category $CHM_\bbQ$ whose objects are the smooth projective $k$-varieties. Hence, we define $\theta_2$ to be the pseudo-abelianization of the auxiliar functor $\widetilde{\theta_2}:KM_\bbQ \to CHM_\bbQ$, $X \mapsto X$, defined as follows
\begin{eqnarray*}
K_0(X\times Y)_\bbQ \too  \bigoplus_{j \in \bbZ} \Corr^j(X,Y) &&
\alpha \mapsto  ch(\alpha) \cdot \pi_Y^\ast(Td(Y))\,,
\end{eqnarray*}
where $ch$ denotes the Chern character and $Td$ the Todd genus. Note that the Grothendieck-Riemann-Roch theorem guarantees that this functor $\widetilde{\theta_2}$ is well-defined; see~\cite[page~39]{GS1}. Moreover, it is fully-faithful and induces a bijection on objects. Hence, the associated functor $\theta_2$ is an equivalence of categories. Finally, the Grothendieck-Riemann-Roch theorem and the fact that the functors \eqref{eq:func-Manin} and \eqref{eq:func-GS} factor through $CHM_\bbQ$ and $KM_\bbQ$, respectively, allow us to conclude that the middle rectangle of the above diagram is commutative. Note also that by construction $\theta_2$ is $\bbQ$-linear, additive and symmetric monoidal.

Let us now consider the functor $\theta_3$. Recall from \S\ref{sub:GS} and \S\ref{sub-idemp} that we have a pseudo-abelianization functor $KM_\bbQ \to \KM_\bbQ$ and an idempotent completion functor $\Mix_\bbQ \to \Mix_\bbQ^\natural$. We will then construct an auxiliar functor $\widetilde{\theta_3}: KM_\bbQ \to \Mix_\bbQ$ and define $\theta_3$ as the canonical extension of the composed functor 
$$KM_\bbQ \stackrel{\widetilde{\theta_3}}{\too} \Mix_\bbQ \stackrel{(-)^\natural}{\too} \Mix_\bbQ^\natural$$
to $\KM_\bbQ$. Given smooth projective $k$-varieties $X$ and $Y$ we have the following natural isomorphisms: 
\begin{eqnarray}
\Hom_{\Mix_\bbQ}(U(\perf(X))_\bbQ, U(\perf(Y))_\bbQ) & =  & 
K_0((\perf(X))^\op \otimes \perf(Y))_\bbQ \nonumber \\
& \simeq & K_0(\perf(X) \otimes \perf(Y))_\bbQ \label{eq:star2} \\
& \simeq & K_0(\perf(X\times Y))_\bbQ \label{eq:star3} \\
& \simeq & K_0(X \times Y)_\bbQ \,. \label{eq:star4}
\end{eqnarray}
Isomorphism \eqref{eq:star2} follows from the fact that the dg category $\perf(X)$ is self-dual, Isomorphism \eqref{eq:star3} from the natural isomorphism $\perf(X) \otimes \perf(Y) \simeq \perf(X\times Y)$ in $\Hmo$, and Isomorphism \eqref{eq:star4} from the classical fact that the Grothendieck group of an algebraic variety can be recovered from its derived category of perfect complexes. Under these isomorphisms, composition in $\Mix_\bbQ$ corresponds to
$$
\begin{array}{ccc}
K_0(X \times Y)_\bbQ \times K_0(Y\times Z)_\bbQ  & \too &  K_0(X \times Z)_\bbQ \\
([\cF], [\cG]) & \mapsto & [\cF \otimes^\bbL_{\cO_Y} \cG]=\sum_{i \geq 0} (-1)^i [\mathrm{Tor}_i^{\cO_Y}(\cF,\cG)]\,.
\end{array}
$$
Hence, we define $\widetilde{\theta_3}$ as the functor that sends a smooth projective $k$-variety $X$ to $U(\perf(X)_\bbQ)$ and which is the identity on morphisms (via the isomorphisms \eqref{eq:star2}-\eqref{eq:star4}). By construction, we obtain then a fully-faithful functor $\theta_3$ which makes the right rectangle in the above diagram commutative. Moreover, it is $\bbQ$-linear, additive and symmetric monoidal.

Finally, choose inverse functors $\theta_1^{-1}$ and $\theta_2^{-1}$ to $\theta_1$ and $\theta_2$ and define $R$ as the composition $\theta_3 \circ \theta_2^{-1}\circ \theta_1^{-1}$. By construction, $R$ is fully-faithful and makes the diagram \eqref{eq:diag-main} commutative. Finally, the fact that it is moreover $\bbQ$-linear, additive and symmetric monoidal follows from the corresponding properties of the functors $\theta_1$, $\theta_2$ and $\theta_3$.
\end{proof}
\section{Remaining proofs}\label{sec:proofs}
\subsection*{Theorem~\ref{thm:2}}
Recall from \cite{Deligne} that Schur and Kimura finiteness is preserved by symmetric monoidal $\bbQ$-linear functors. Hence, item (i) follows from the commutativity of diagram~\eqref{eq:diag-main} and from the fact 
that both functors $\pi$ and $R$ are $\bbQ$-linear and symmetric monoidal. Item (ii) follows 
from the faithfulness of $\pi$ and $R$.
\subsection*{Theorem~\ref{thm:3}}
Since the functors $\pi$ and $R$ in diagram \eqref{eq:diag-main} are both additive and symmetric monoidal, their composite gives rise to a ring homomorphism
\begin{equation}\label{eq:ring}
K_0(\Chow_\bbQ) \too K_0(\Mix_\bbQ^\natural)\,.
\end{equation}
Recall from \cite{GS1} that Gillet-Soul{\'e}'s motivic measure $\mu_{GS}$ sends a smooth projective $k$-variety $X$ to $[M(X)]$. Hence, by composing \eqref{eq:ring} with $\mu_{GS}$ we obtain a well-defined motivic measure $\mu_{NC}$ sending a smooth projective $k$-variety $X$ to $[NC(X)]$.
\subsection*{Example~\ref{ex:main}}
Recall from \cite[\S8.4]{CT1} that Hochschild homology $HH$ gives rise to a $\otimes$-triangulated functor $\Mix \to \cD_c(\bbC)$. Since $\cD_c(\bbC)$ is idempotent complete and $\bbQ$-linear, this functor extends uniquely to a $\otimes$-triangulated functor $\overline{HH}: \Mix_\bbQ \to \cD_c(\bbC)$ and so it induces a ring homomorphism
\begin{equation}\label{eq:homo-1}
K_0(\Mix_\bbQ^\natural) \to K_0(\cD_c(\bbC)) =\K_0(\bbC) =\bbZ\,.
\end{equation} 
The natural equalities
$$ [\overline{HH}(NC(X))] =[HH(\perf(X))]=[HH(X)] = \sum_{i \geq 0}(-1)^i \mathrm{dim}\,H^i(X, \Omega_X)\,,$$
for every smooth projective $\bbC$-variety $X$, combined with Bittner-Looijenga's presentation of $K_0(\Var_\bbC)$ (see \cite{Bittner,Looijenga}), allow us then to conclude that $\mu_{\chi_c}$ factors through $\mu_{NC}$ via the ring homomorphism \eqref{eq:homo-1}.
\subsection*{Theorem~\ref{thm:4}}
By definition of the formal power series $\zeta(NC(X);t)$ and $\zeta_{\mu_{NC}}(X;t)$, it suffices to show the equality
\begin{equation}\label{eq:1}
[\Sym^nNC(X)] = \mu_{NC}([S^n(X)]) \in K_0(\Mix_\bbQ^\natural)
\end{equation}
for every natural number $n$. A similar equality
\begin{equation}\label{eq:2}
[\Sym^nM(X)]=\mu_{GS}([S^n(X)]) \in K_0(\Chow_\bbQ)
\end{equation}
in the setting of Chow motives was proved by Ba{\~n}o-Aznar in \cite{BN}. Since the functors $\pi$ and $R$ in diagram \eqref{eq:diag-main} are both symmetric monoidal and $\bbQ$-linear, their composition gives rise to a ring homomorphism
\begin{equation}\label{eq:ultimo}
\K_0(\Chow_\bbQ) \too K_0(\Mix_\bbQ^\natural)
\end{equation}
which maps $[\Sym^nM(X)]$ to $[\Sym^nNC(X)]$. Hence, by applying this ring homomorphism \eqref{eq:ultimo} to the equalities \eqref{eq:2} we obtain the searched equalities \eqref{eq:1}.

\medbreak

\noindent\textbf{Acknowledgments:} The author is very grateful to Alexander Beilinson, Denis-Charles Cisinski, Maxim Kontsevich, Haynes Miller, Carlos Simpson, and Bertrand To{\"e}n for stimulating conversations. He would also like to thank the Department of Mathematics of MIT for its hospitality and excellent working conditions.

\end{document}